\documentclass{amsart}
\title[GIT stability and biquotients of $SU(3)$]
{GIT stability and biquotients of $SU(3)$}
\author{Yoshinori Hashimoto \\ Hiroaki Ishida \\ Hisashi Kasuya}
\address[Yoshinori Hashimoto]{Department of Mathematics, Graduate School of Science, Osaka Metropolitan University}
\email{yhashimoto@omu.ac.jp}
\address[Hiroaki Ishida]{Department of Mathematics, Graduate School of Science, Osaka Metropolitan University}
\email{hiroaki.ishida@omu.ac.jp}
\address[Hisashi Kasuya]{Graduate School of Mathematics,
Nagoya University}
\email{kasuya@math.nagoya-u.ac.jp}
\date{\today}


\usepackage{amssymb}
\usepackage{amsmath}
\usepackage{amscd}
\usepackage{amstext}
\usepackage{amsfonts}
\usepackage{amsrefs}
\usepackage{color}
\usepackage{tikz}
\usepackage{tikz-cd}
\usepackage{graphicx}

\newif\ifdebug

\debugfalse

\newtheorem{prop}{Proposition}[section]
\newtheorem*{prop*}{Proposition \referenza}
\newtheorem{theo}[prop]{Theorem}
\newtheorem*{theo*}{Theorem \referenza}
\newtheorem{coro}[prop]{Corollary}
\newtheorem*{coro*}{Corollary}
\newtheorem{lemm}[prop]{Lemma}
\newtheorem*{cla}{Claim} 
\newtheorem*{maintheo}{Main Theorem}

\theoremstyle{remark}
\newtheorem{rema}[prop]{Remark}

\theoremstyle{definition}
\newtheorem{defi}[prop]{Definition}
\newcommand{\C}{\mathbb{C}}
\newcommand{\R}{\mathbb{R}}
\newcommand{\N}{\mathbb{N}}

\newcommand{\Z}{\mathbb{Z}}

\newcommand{\diag}{\mathrm{diag}}

\newcommand{\cone}{\mathrm{cone}}

\newcommand{\conv}{\mathrm{conv}}

\newcommand{\Proj}{\mathrm{Proj}}
\newcommand{\chiss}{\chi\mathchar`-ss}

\newcommand{\chis}{\chi\mathchar`-s}
\DeclareMathOperator{\Int}{\mathrm{Int}}

\newcommand{\git}{\mathbin{
  \mathchoice{/\mkern-6mu/}
    {/\mkern-6mu/}
    {/\mkern-5mu/}
    {/\mkern-5mu/}}}

\makeatletter

\@addtoreset{equation}{section}
\makeatother

\def\co{\colon\thinspace}

\begin{document} 

\begin{abstract}
We study double-sided actions of $(\C^*)^2$ on $SL(3,\C)/U$ and the associated quotients, where $U$ is a maximal unipotent subgroup of $SL(3,\C)$. The main results of this paper are a sufficient condition for the double-sided quotient to   coincide with the GIT quotient, and an explicit necessary and sufficient condition for $SL(3,\C)/U$ to coincide with the $\chi$-stable locus in its affine closure. We apply this result to characterize certain complex structures on $SU(3)$ which are not left invariant by means of the GIT quotient. 
\end{abstract}

\maketitle

\section{Introduction}
We begin by recalling the definition of  a flag variety. A flag in $\C^n$ is a sequence 
\begin{equation*}
	0 = V_0 \subset V_1 \subset \dots \subset V_{n-1} \subset V_n = \C^n
\end{equation*}
of vector subspaces of $\C^n$ with $\dim V_k = k$ for all $k =0, 1,\dots, n$. 
The flag variety $\mathrm{Flag}(\C^n)$ consists of all flags in $\C^n$. The natural action of the special linear group $SL(n,\C)$ on $\C^n$ induces a transitive action on $\mathrm{Flag}(\C^n)$. The isotropy subgroup at the standard flag 
\begin{equation*}
	0 \subset \langle e_1\rangle \subset \langle e_1,e_2\rangle \subset \dots \subset \langle e_1,\dots, e_{n-1}\rangle \subset \langle e_1,\dots, e_n \rangle = \C^n 
\end{equation*}
is the subgroup $B$ consisting of all upper triangular matrices in $SL(n,\C)$. The subgroup $B$ is a Borel subgroup of $SL(n,\C)$, and $\mathrm{Flag}(\C^n)$ has a structure of the homogeneous space $SL(n,\C)/B$. By choosing a suitable character $B \to \C^*$, we obtain an ample line bundle $SL(n,\C) \times_B \C$, where the action of $B$ on $\C$ is given by the chosen character. 
Let $U$ be the commutator subgroup of $B$. Let $H$ be the algebraic torus consisting of all diagonal matrices in $SL(n,\C)$. Then, $B$ can be decomposed into the semi-direct product as $B = H \ltimes U$.  Therefore, $\mathrm{Flag} (\C^n)$ may be regarded as a quotient of $SL(n,\C)/U$ by an action of the algebraic torus $(\C^*)^{n-1}$. 
On the other hand, the algebraic torus $H \times H \cong (\C^*)^{2n-2}$ acts on $SL(n,\C)/U$ by ``double-sided'' multiplications. For $(g_L, g_R) \in H \times H$ and $[A] \in SL(n,\C)/U$, the action is given by $(g_L, g_R) \cdot [A] = [g_LAg_R^{-1}]$. In this paper, we consider quotients of $SL(n,\C)/U$ by actions of $H \times H$ restricted to an $(n-1)$-dimensional algebraic subtorus $(\C^*)^{n-1} \to H \times H$ and a linear character $\chi \co (\C^*)^{n-1} \to \C^*$. We seek conditions on $\chi$ under which the quotient is well behaved.

In this paper, we restrict our attention to the case $n = 3$, for which a more concrete description is available as follows. In this case, we can embed  $SL(3,\C)/U$ into $\C^3 \times \C^3$ as an orbit of the action of $SL(3,\C)$ on $\C^3 \times \C^3$ associated with the standard representation of $SL(3,\C)$ and its dual representation.
Under this embedding, 
$SL(3,\C)/U$ equipped with the action of a subtorus $(\C^*)^2 \to H \times H$ is equivariantly isomorphic (up to roots of unity) to the quasi-affine variety 
\begin{equation*}
	M = \left\{ (z,w) \in \C^3 \times \C^3 \mathrel{}\middle|\mathrel{} z \neq 0, w\neq 0, \sum_{i=1}^3z_iw_i = 0\right\}, 
\end{equation*}
where a double-sided action of the $2$-dimensional algebraic torus $(\C^*)^2$ on $M$ is given by
\begin{equation*}
	g \cdot (z,w) = (g^{A_1}z_1, g^{A_2}z_2, g^{A_3}z_3, g^{B_1}w_1, g^{B_2}w_2, g^{B_3}w_3), \quad g \in (\C^*)^2, (z,w) \in M
\end{equation*}
for certain weights $A_1,\dots, B_3 \in \Z^2$ (see Section \ref{sec:torusaction} and \cite[Section 3]{IK2025} for the details). The main result of this paper is the following, which gives a necessary and sufficient condition on $\chi$ so that the quotient by the $(\C^*)^2$-action agrees with the one in terms of the geometric invariant theory (GIT).

\begin{maintheo}[Theorem \ref{theo:maintheo}]
	Suppose that we have a double-sided action of $(\C^*)^2$ on $M$ whose weight is given by $A_1, A_2, A_3, B_1, B_2, B_3 \in \Z^2$, and that we choose a nontrivial linear character $\chi \co (\C^*)^2 \to \C^*$ whose weight is given by $\chi (g) = g^C$ for $C \in \Z^2$. Let $\overline{M}$ be the $(\C^*)^2$-equivariant affine closure of $M$, and $\overline{M}^{\chis}$ the $\chi$-stable locus in $\overline{M}$. Then, $\chi$ satisfies $M = \overline{M}^{\chis}$ if and only if it satisfies the following  ``Japanese fan'' condition: 
	\begin{itemize}
		\item[($\star$)] $\cone (A_1,A_2,A_3,B_1,B_2,B_3)$ has an apex at the origin and $C \in \Int \cone (A_i,B_j)$ for all  $i,j$ with $i \neq j$.  
	\end{itemize}
	Moreover, when the condition ($\star$) is satisfied, the quotient topological space $M / (\C^*)^2$ is a complex analytic space isomorphic  to the GIT quotient $\overline{M} \git_\chi (\C^*)^2$, which is a projective variety.
\end{maintheo}

\begin{figure}[h]
\begin{minipage}[b]{0.49\columnwidth}
\centering
\begin{tikzpicture}[scale=2.5, >=stealth]

  \coordinate (O) at (0,-0.5);

  \coordinate (A1) at (-0.8,0.6);
  \coordinate (A2) at (-0.4,0.7);
  \coordinate (A3) at (-1,0.4);

  \coordinate (B1) at (0.8,0.6);
  \coordinate (B2) at (0.4,0.7);
  \coordinate (B3) at (1,0.4);

  \coordinate (C) at (0, 0.3);


  \foreach \pt/\name in {A1/A_1, A2/A_2, A3/A_3, B1/B_1, B2/B_2, B3/B_3} {
    \draw[->, thick] (O) -- (\pt) node[pos=1.1] {$\name$};
  }

  \filldraw[black] (C) circle (1pt) node[above right] {$C$};

  \filldraw[black] (O) circle (1pt);
  \node[below left] at (O) {$0$};
  
\end{tikzpicture}
\caption{Illustration of the condition ($\star$).}
\label{fig:jpnsfan}
\end{minipage}
\begin{minipage}[b]{0.49\columnwidth}
\centering
\includegraphics[width=5cm]{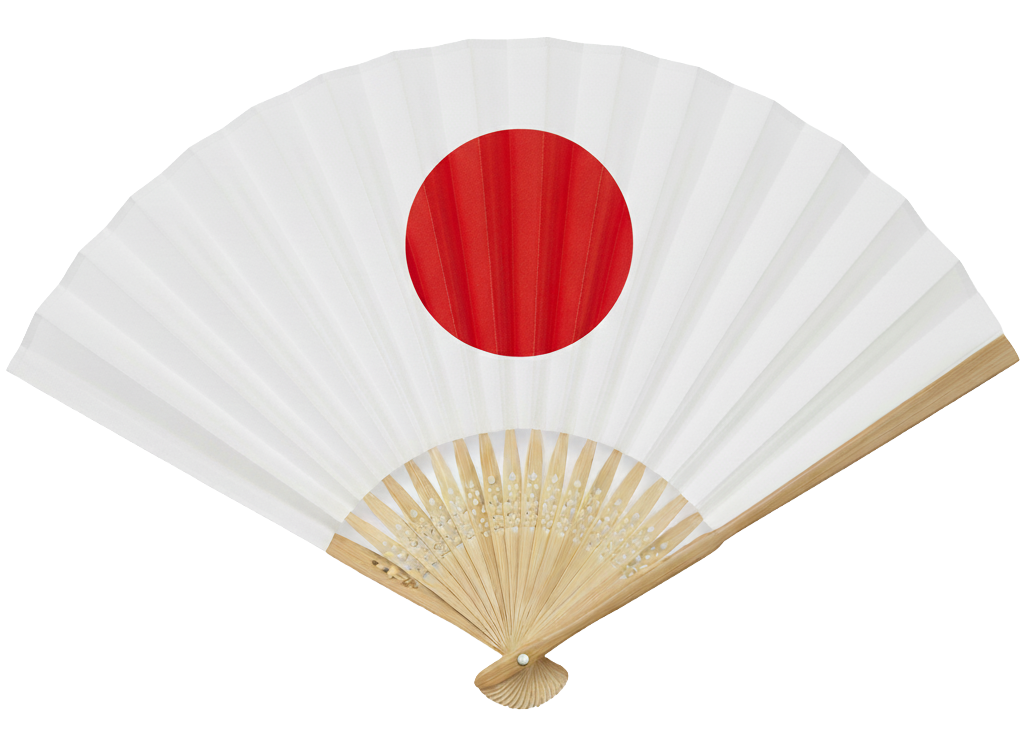}
\caption{Japanese traditional fan \textit{`sensu'}}
\end{minipage}
\end{figure}

All terminologies in the Main Theorem are explained in Section 2. When we only consider the natural right action of $H$ on $SL(n,\C)/U$, it is a standard fact that its quotient by $H$ is $\mathrm{Flag}(\C^3)$, which can be likewise constructed as a GIT quotient. The theorem above generalizes this result to double-sided actions, since it is easy to check that the right action satisfies the condition ($\star$). The choice of $\chi$ in the above can be interpreted as a choice of an ample line bundle (polarization) in the classical GIT.

As an application of the Main Theorem, we give a GIT characterization for the existence of complex structures on $SU(3)$  that is not left-invariant, as constructed by the last two authors \cite{IK2025}. Let $T$ be the diagonal torus in $SU(3)$ and  $\rho_L, \rho_R \co (S^1)^2 \to T$ smooth homomorphisms given by \begin{equation*}
		\begin{split}
			\rho_L(t) &= \diag (t^{w_1^L}, t^{w_2^L},t^{w_3^L}) \\
			\rho_R(t) &= \diag (t^{w_1^R}, t^{w_2^R},t^{w_3^R})
		\end{split}
	\end{equation*}
	for $t \in (S^1)^2$, where $w_j^L, w_j^R \in \Z^2$. Put 
	\begin{equation*}
		A_j := w_j^L-w_1^R, \quad B_j := -w_j^L + w_3^R, \quad C := -w_1^R + w_3^R.
	\end{equation*}
According to \cite{IK2025}, if the condition $(\star)$ is fulfilled, then we can construct a $T \times T$-invariant complex structure on $SU(3)$ obtained as follows: we construct a moment map $\Phi \co M \to \R^2$ with respect to the $(S^1)^2$-action, and have a regular level set $\Phi ^{-1}(C) \subset M$ which is equivariantly diffeomorphic to $SU(3)$. On the other hand, we can find a holomorphic foliation on a neighborhood of $\Phi ^{-1}(C)$ which is transverse to $\Phi ^{-1}(C)$. By Haefliger's trick \cite{Haefliger1985}, we obtain a complex structure on $SU(3)$ via a diffeomorphism to $\Phi ^{-1}(C)$. It is not obvious that the complex structure on $SU(3)$ is a quotient of the whole space $M$ by an action of $\C$. 
 The quotient $SU(3)/(\rho_L,\rho_R)((S^1)^2)$ also has a $T \times T$-invariant K\"ahler orbifold structure 
 such that the natural projection $\pi \co SU(3) \to SU(3)/(\rho_L,\rho_R)((S^1)^2)$ is holomorphic. It is natural to ask whether $SU(3)/(\rho_L,\rho_R)((S^1)^2)$ equipped with the K\"ahler orbifold structure is projective or not. 
 \begin{coro*}[Corollary \ref{coro}] 
	Assume that $A_j, B_j$ for $j=1,2,3$ and $C$ satisfy the condition $(\star)$. Then, the following hold: 
	\begin{enumerate} 
		\item $SU(3)$ equipped with the above complex structure is biholomorphic to the quotient of $M$ by a free action of $\C$. 
		\item $SU(3)/(\rho_L,\rho_R)((S^1)^2)$ is isomorphic to $\overline{M} \git_\chi (\C^*)^2$ as analytic spaces. In particular, $SU(3)/(\rho_L,\rho_R)((S^1)^2)$ has a structure of a projective variety. 
	\end{enumerate}
\end{coro*}
We remark that the action of $\C$ on $M$ is not algebraic but holomorphic, and $SU(3)$ does not have a K\"ahler structure for topological reasons. We will explain the details in Section 4. The above complex structure on $SU(3)$ is a Lie group analogue of the manifolds constructed in \cite{LV1997}. These manifolds are generalized to the so-called LVM manifolds, which are constructed in \cite{Meersseman2000}. The holomorphic map $\pi \co SU(3) \to SU(3)/(\rho_L,\rho_R)((S^1)^2)$ is a Lie group analogue of the holomorphic Seifert fibering over toric varieties, shown in \cite{MV2004}. When we consider the natural right action of $T$ on $SU(3)$, the complex structure on $SU(3)$ is left-invariant. If $\rho_L$ is non-trivial, then the complex structure of $SU(3)$ in the Corollary is not left-invariant. Non-left-invariant complex structures on compact Lie groups are studied in \cite{IK2020} and \cite{LMN2007}.

The complex and symplectic structures on a biquotient of $SU(3)$ which differ from the objects discussed in this paper, are studied in \cite{GKZ2020}. Our construction is closely related to the GIT quotients that appear in the theory of Cox rings, as explained in Remark \ref{rmcxrg} (see \cite{AH2009,ADJ2013,ADHL} for more details). Other related works are \cite{Anisimov2012-1,Anisimov2012-2} that deal with GIT constructions of double coset varieties.

\bigskip

\noindent {\bf Organization.} 
Section 2 is devoted to the preliminaries for the latter sections. In Section 3, we characterize $\chi$-stable points in terms of cones in $\R^2$. In Section 4, we show the Main Theorem and give an application to complex structures on biquotients of $SU(3)$. In Appendix, we explain the universal property of quotients of complex manifolds by locally free proper actions of complex Lie groups. This result is likely well-known to the experts, but we give a detailed proof for the reader's convenience.

\bigskip 

\noindent{\bf Acknowledgements.} The authors would like to thank Ivan Arzhantsev for his valuable comments.

\section{Preliminaries}
\subsection{Rational polyhedral cones}
	Let $V = \{v_1,\dots, v_m\}$ be a finite subset of vectors in $\R^n$. A polyhedral cone generated by $V$ is the set of all linear combinations of elements in $V$ with nonnegative coefficients. Namely, 
	\begin{equation*}
		\cone (V) = \{ a_1v_k + \dots + a_mv_m \mid a _i \geq 0\}. 
	\end{equation*}
	If each $v_i$ is a rational vector, $\cone (V)$ is said to be rational polyhedral cone. We remark that if $C$ is a rational polyhedral cone $\cone (V)$ for some $V$, we may assume that each element in $V$ is integral, that is, $v_i \in \Z^n$. 
	
	If there exists a nonzero linear map $f \co \R^n \to \R$ such that 
	\begin{itemize}
		\item $H^+ := \{ v \in \R^n \mid f(v)\geq 0\} \supset \cone (V)$ and 
		\item $\cone (V) \cap \{ v \in \R^n \mid f(v) = 0\} = \{0\}$,
	\end{itemize}
	then we say that $\cone (V)$ has an apex  at the origin. The polyhedral cone $\cone (V)$ has an apex at the origin if and only if there exists a linear function $f \co \R^n \to \R$ such that $f(v_i)>0$ unless $v_i = 0$. If $\cone (V)$ is rational and has an apex at the origin, then we may choose $f$ above to be integral. 
	
	We need the following lemma for later use. 
	\begin{lemm}\label{lemm:Intcone}
		Let $A, B, C$ be vectors in $\R^2$ such that $C \in \cone (A,B)$ but $C \notin \cone (A) \cup \cone (B)$. Then $C \in \Int \cone (A,B)$. 
	\end{lemm}
	\begin{proof}
		Suppose that $C = aA+ bB$ for some $a, b \geq 0$. If $a = 0$, then $C \in \cone (B)$ because $C = bB$. This contradicts the assumption $C \notin \cone (A) \cup \cone (B)$. By the same argument, we have $a>0$ and $b>0$.  
		
		We claim that $A$ and $B$ are linearly independent. Suppose $a'A + b'B = 0$ for some $a',b' \in \R$ with $(a',b') \neq (0,0)$. If $(a',b')$ is a multiple of $(a,b)$, then $C=0 \in \cone (A) \cup \cone (B)$, contradiction. Thus $(a',b')$ is not a multiple of $(a,b)$. Thus we have 
		\begin{equation*}
			A = \frac{b'}{ab'-a'b}C, \quad B = \frac{-a'}{ab'-a'b}C. 
		\end{equation*}
		Since $C \notin \cone (A) \cup \cone (B)$, we have $\frac{b'}{ab'-a'b}<0$ and $\frac{-a'}{ab'-a'b}<0$. 
		In particular, $a'$ and $b'$ have different signs. 
		By eliminating $A$ ($B$, respectively) from $C = aA+bB$ and $0=a'A+b'B$ in the case when $a'<0$ ($b'<0$, respectively), we have $C \in \cone (B)$ ($\cone (A)$, respectively), contradiction. 
		Therefore $A$ and $B$ are linearly independent. 
		
		Since $A$ and $B$ form a basis of $\R^2$, we have 
		\begin{equation*}
			\Int \cone(A,B) = \{ \alpha A + \beta B \mid \alpha, \beta >0\}. 
		\end{equation*}
		Therefore $C \in \Int \cone(A,B)$, proving the lemma.
	\end{proof}
	
\subsection{The quasi-affine variety $M$ and its closure $\overline{M}$}
Let $M$ be the quasi-affine variety given by
\begin{equation*}
	M = \left\{ (z,w)= (z_1,z_2,z_3,w_1,w_2,w_3) \in \C^3 \times \C^3 \mathrel{}\middle|\mathrel{} z \neq 0, w \neq 0, \sum_{i=1}^3 z_iw_i=0\right\}. 
\end{equation*}
Let $\overline{M}$ be the closure of $M$ in $\C^6$. Namely, $\overline{M}$ is the affine variety 
\begin{equation*}
	\overline{M} = \left\{(z,w)= (z_1,z_2,z_3,w_1,w_2,w_3) \in \C^3 \times \C^3 \mathrel{}\middle|\mathrel{} \sum_{i=1}^3 z_iw_i=0\right\}
\end{equation*}
embedded in $\C^3 \times \C^3$.

\begin{rema} \label{rmcxrg}
	The manifolds $M$ and $\overline{M}$ are closely related to the Cox ring of the flag variety \cite{ADHL}, as follows. Let $B$ be the Borel subgroup of $SL(3, \mathbb{C})$ (which is the complexification of $SU (3)$), and let $X := SL(3,\mathbb{C}) / B$ be the associated flag variety.
	
	The Cox sheaf is defined as
	\begin{equation*}
		\mathcal{R}:= \bigoplus_{[D] \in \mathrm{Cl} (X)} \mathcal{O}_X(D),
	\end{equation*}
	and the Cox ring $\mathcal{R} (X)$ is the algebra of global sections of $\mathcal{R}$ \cite[\S 1.4.1]{ADHL}. It is well-known that $\mathcal{R} (X)$ is finitely generated over $\mathbb{C}$ when $X$ is a flag variety. The characteristic space of the Cox sheaf is defined to be the relative spectrum $\mathrm{Spec}_X (\mathcal{R})$ over $X$, and the total coordinate space is defined to be the affine variety $\mathrm{Spec} ( \mathcal{R} (X) )$. It is well-known that the total coordinate space of $X = SL(3,\mathbb{C}) / B$ is given by $\overline{M}$, which in turn agrees with the affine cone over $X$ \cite[\S 3.2.3]{ADHL}. We also find that the characteristic space agrees with $M$, and that $X$ is obtained as the GIT quotient of the total coordinate space by the torus action, as explained in \cite[\S 1.6]{ADHL}.  
\end{rema}

\subsection{Torus actions and the coordinate rings}\label{sec:torusaction}

Let $A_1,A_2,A_3,B_1,B_2,B_3 \in \Z^2$  be such that $A_1 +B_1 = A_2+B_2 = A_3+B_3$. We equip an action of $(\C^*)^2$ on $\overline{M}$ as follows. For $g = (g_1,g_2) \in (\C^*)^2$ and $(z,w) \in \overline{M}$, 
\begin{equation}\label{eq:algtorusactiononM}
	g\cdot (z,w):= (g^{A_1}z_1, g^{A_2}z_2, g^{A_3}z_3, g^{B_1}w_1, g^{B_2}w_2, g^{B_3}w_3),
\end{equation}
here we use the multi-index notation. Namely, for $A = (a_1,a_2) \in \Z^2$ by $g^{A}$ we mean $g_1^{a_1}g_2^{a_2}$.

The coordinate ring $R$ of $\overline{M}$ is represented as 
\begin{equation*}
	R = \C [z_1,z_2,z_3,w_1,w_2,w_3]\left/\left(\sum_{i=1}^3z_iw_i\right)\right. .
\end{equation*}
The action of $(\C^*)^2$ on $\overline{M}$ induces a natural right action on $R$ as follows. For $g \in (\C^*)^2$, $f \in R$ and $(z,w) \in M$, 
\begin{equation*}
	f^g (z,w) := f(g\cdot (z,w)). 
\end{equation*}
In particular, for the generators $z_1,\dots, w_3 \in R$ we have $z_i^g = g^{A_i}z_i$ and $w_j^g = g^{B_j}w_j$ for $i,j=1,2,3$. We take a linear character $\chi \co (\C^*)^2 \to \C^*$ and define a graded subspace $R^\chi_n$ of $R$ by 
\begin{equation*}
	R_n^\chi := \{ f \in R \mid f^g = \chi(g)^n f\}
\end{equation*}
for $n \in \N \cup \{0\}$. We assume throughout that $\chi$ is nontrivial. The direct sum $R^\chi := \bigoplus _{n \in \N \cup\{0\}} R^\chi_n$ forms a graded algebra.

We note that the subspace $R_0^\chi$ is the set of $(\C^*)^2$-invariant elements and hence does not depend on the choice of $\chi$.

\begin{prop} \label{prop:dgzrc}
	$R_0^\chi = \C$ if and only if $A_i \neq 0$ and $B_j\neq 0$ for all $i,j$ and $\cone (A_1,A_2,A_3,B_1,B_2,B_3)$ has an apex at the origin. 
\end{prop}

\begin{proof}
	By definition of $R^\chi_0$, it is generated by monomials of the form $z_1^{k_1}z_2^{k_2}z_3^{k_3}w_1^{l_1}w_2^{l_2}w_3^{l_3}$ such that 
	\begin{equation*}
		(z_1^{k_1}z_2^{k_2}z_3^{k_3}w_1^{l_1}w_2^{l_2}w_3^{l_3})^g = 1
	\end{equation*}
	for all $g \in (\C^*)^2$. On the other hand, the weight of $g$ is given by
	\begin{equation*}
		(z_1^{k_1}z_2^{k_2}z_3^{k_3}w_1^{l_1}w_2^{l_2}w_3^{l_3})^g = g^{\sum_{i=1}^3k_iA_i + \sum_{j=1}^3l_jB_j} z_1^{k_1}z_2^{k_2}z_3^{k_3}w_1^{l_1}w_2^{l_2}w_3^{l_3}. 
	\end{equation*}
	Since any non-constant element in $R^\chi_0$ is a $\mathbb{C}$-linear combination of monomials whose weights are (term-wise) trivial, we find that there exists a non-constant element in $R^\chi_0$ if and only if $0$ is a nontrivial conical combination of $A_1,A_2,A_3,B_1,B_2,B_3$. Thus $R_0^\chi = \C$ if and only if $\sum_{i=1}^3k_iA_i + \sum_{j=1}^3l_jB_j=0$ implies $k_i=l_i=0$ for $i=1,2,3$, which in turn happens if and only if $A_i \neq 0$ and $B_j\neq 0$ for all $i,j$ and $\cone (A_1,A_2,A_3,B_1,B_2,B_3)$ has an apex at the origin.
\end{proof}

\subsection{$\chi$-stability and Hilbert-Mumford criterion}\label{sec:HMC}
We define $\overline{M} \git_\chi (\C^*)^2$ to be the scheme $\Proj (R^\chi)$ projective over $\mathrm{Spec} (R_0^\chi )$, which is a projective scheme in the usual sense if $R_0^\chi = \C$ holds as in the above proposition. Recall the following definition \cite[Definition 2.1]{King94}. Let $(\C^*)^2$ act on $\overline{M} \times \C$ by 
	\begin{equation*}
		g\cdot ((z,w),v) := (g\cdot (z,w), \chi(g)^{-1}v)
	\end{equation*}
	for $(z,w) \in M$ and $v \in \C$. Let $v \neq 0$. 
	\begin{enumerate}
		\item $(z,w) \in \overline{M}$ is said to be $\chi$-semistable if there exists $f \in R^\chi_n$ for some $n \ge 1$ such that $f(z,w) \neq 0$.   We denote by $\overline{M}^{\chiss}$ the set of $\chi$-semistable points in $\overline{M}$.
		\item $(z,w) \in \overline{M}$ is said to be $\chi$-stable if there exists $f \in R^\chi_n$ for some $n \ge 1$ such that $f(z,w) \neq 0$, its isotropy subgroup is finite, and the $(\C^*)^2$-action on $\{ (z,w) \in \overline{M} \mid f(z,w) \neq 0 \}$ is closed. We denote by $\overline{M}^{\chis}$ the set of $\chi$-stable points in $\overline{M}$.
	\end{enumerate}

By \cite[Lemma 2.2]{King94}, we find that $(z,w) \in \overline{M}$ is $\chi$-stable if and only if the $(\C^*)^2$-orbit of $((z,w),v) \in \overline{M} \times \mathbb{C}$ is closed (in Zariski topology) and its isotropy group is finite. By \cite[Lemma 3.17]{Newstead78}, we further find that $(z,w) \in \overline{M}$ is $\chi$-stable if and only if the action morphism
\begin{equation} \label{eqactmrpph}
	\sigma : (\C^*)^2 \ni g \mapsto g \cdot ((z,w),v) \in \mathbb{C}^3 \times \mathbb{C}^3 \times \mathbb{C}
\end{equation}
is proper.

	Let $\alpha =(\alpha_1,\alpha_2) \in \Z^2$ and $\lambda \co \C^* \to (\C^*)^2$ the one parameter subgroup ($1$-PS for short) given by $\lambda(h) = (h^{\alpha_1}, h^{\alpha_2})$ for $h \in \C^*$. Let $(z,w) \in \overline{M}$. We put 
	\begin{equation} \label{dfeqhmwgt}
		\mu_\chi(\lambda, (z,w)) = -\min(\{\langle A_i,\alpha\rangle \mid z_i \neq 0\} \cup \{\langle B_j,\alpha\rangle \mid w_j \neq 0\} \cup \{-\langle C, \alpha\rangle\}). 
	\end{equation} 
	Then, we have the following criterion called the Hilbert--Mumford criterion: Let $(z,w) \in \overline{M}$. 
	\begin{enumerate}
		\item $(z,w) \in \overline{M}^{\text{$\chi$-ss}}$ if and only if $\mu_\chi (\lambda, (z,w)) \geq 0$ for all $1$-PS $\lambda$. 
		\item $(z,w) \in \overline{M}^{\chis}$ if and only if $\mu_\chi (\lambda, (z,w)) \geq 0$ for all $1$-PS $\lambda$ and equality holds if and only if the subgroup $\lambda(\C^*)$ is trivial. 
	\end{enumerate}
	A proof of the above can be found in a standard text book in Geometric Invariant Theory (see e.g.~\cite{King94}, \cite{MFK}, and \cite{Nakajima99}).

In general, for a complex reductive linear algebraic group $G$ acting on $\C^m$, we consider a 1-PS $\xi \co \C^* \to G$ acting with the weights
\begin{equation*}
	\xi (h) \cdot (x_1 , \dots , x_m) = (h^{\beta_1} x_1 , \dots , h^{\beta_m} x_m), \quad \beta_1 , \dots , \beta_m \in \Z.
\end{equation*}
We define the Hilbert--Mumford weight by
\begin{equation*}
	\mu (\xi , x) := - \min \{ \beta_i \mid 1 \le i \le m \text{ with } x_i \neq 0 \}.
\end{equation*}
It is well-known that $\mu (\xi , x) \ge 0$ for any 1-PS $\xi \co \C^* \to G$, with equality if and only if $\xi (\C^* )$ is trivial, if and only if the action morphism $\sigma \co G \ni g \mapsto g \cdot x \in \C^m$ is proper \cite[Propositions 4.7 and 4.8]{Newstead78}. We apply this result to our situation, in which we take $G = (\C^*)^2$ and $x = ((z,w),v)$. The weight of the 1-PS $\lambda \co \C^* \ni h \mapsto (h^{\alpha_1}, h^{\alpha_2}) \in (\C^*)^2$ is given by
\begin{equation*}
	\lambda (h) \cdot ((z,w),v) = (h^{\langle A_1 , \alpha \rangle} z_1, h^{\langle A_2 , \alpha \rangle} z_2 , h^{\langle A_3 , \alpha \rangle} z_3 , h^{\langle B_1 , \alpha \rangle} w_1 , h^{\langle B_2 , \alpha \rangle} w_2 , h^{\langle B_3 , \alpha \rangle} w_3 , h^{-\langle C , \alpha \rangle} v ),
\end{equation*}
where we wrote $\chi (\lambda (h)) = h^{\langle C , \alpha \rangle}$ for some $C \in \Z^2$. Thus, we find that $(z,w) \in \overline{M}^{\chis}$ holds if and only if $\mu_\chi (\lambda, (z,w)) \geq 0$ for all $1$-PS $\lambda$ with equality if and only if $\lambda (\C^*)$ is trivial, since $(z,w) \in \overline{M}^{\chis}$ holds if and only if the action morphism (\ref{eqactmrpph}) is proper.

\section{$\chi$-stability in terms of cones} 
\begin{defi}
	We write $\Sigma := \cone (A_1,A_2,A_3,B_1,B_2,B_3)$, and for any $(z,w) \in \overline{M}^{\text{$\chi$-ss}}$ we define the subcone of $\Sigma$ generated by $A_i$'s such that the variable $z_i$ with the corresponding index is nonzero, and similarly for $B_j$'s; namely
\begin{equation*}
	\sigma_{z,w}:= \cone \left( \{A_i \mid z_i\neq 0\}_{i=1}^3 \cup \{B_j \mid w_j\neq 0\}_{j=1}^3 \right).
\end{equation*}
\end{defi}

We give another criterion for $\chi$-semistability in terms of discrete geometry. Assume that $\mu_\chi(\lambda, (z,w)) \geq 0$ for all $1$-PS $\lambda$. Then for each $1$-PS $\lambda'$ given by $\lambda'(h) = (h^{\alpha_1'}, h^{\alpha_2'})$ with $\alpha' = (\alpha_1',\alpha_2') \in \Z^2$, the set 
\begin{equation*}
	\{\langle A_i,\alpha'\rangle \mid z_i \neq 0\}_{i=1}^3 \cup \{\langle B_j,\alpha'\rangle \mid w_j \neq 0\}_{j=1}^3 \cup \{-\langle C, \alpha'\rangle\}
\end{equation*}
of integers contains at least one non-positive integer. The converse is also true. By replacing $\lambda'$ by $(\lambda')^{-1}$, we also have that the set contains at least one non-negative integer. Remark that $\Sigma$ has the apex  at the origin. With this understood, we have the following:
\begin{lemm}\label{lemm:chiss}
	$(z,w) \in \overline{M}^{\chiss}$ if and only if $C \in \sigma_{z,w}$.
\end{lemm}
This lemma follows from the following general result.

\begin{lemm}\label{lemm:forchiss}
	Suppose that we take $u, v_1,\dots, v_m \in \Z^n \setminus \{ 0 \}$ such that $\cone (v_1,\dots, v_m)$ has an apex  at the origin, and assume that for any $\mathbb{R}$-linear map $f \co \R^n \to \R$ with coefficients in $\mathbb{Q}$, at least one of $f(-u), f(v_1),\dots, f(v_m)$ is non-positive. Then, $u \in \cone (v_1,\dots, v_m)$. 
\end{lemm}
\begin{proof}
	We assume $u \neq 0$, since the claim is obvious otherwise. By replacing $\R^n$ by the $\mathbb{R}$-linear hull of $u, v_1,\dots, v_m$ if necessary, we may assume that $u, v_1,\dots, v_m$ span $\R^n$ over $\mathbb{R}$. We then find, by the supporting hyperplane theorem, that the convex hull $\conv (-u, v_1,\dots, v_m)$ contains $0$ if for any $\mathbb{R}$-linear map $f' \co \R^n \to \R$ at least one of $f'(-u), f'(v_1),\dots, f'(v_m)$ is non-positive. This condition is satisfied if for any $\mathbb{Q}$-linear map $f \co \mathbb{Q}^n \to \mathbb{Q}$ at least one of $f(-u), f(v_1),\dots, f(v_m)$ is non-positive, as stated in the hypothesis of the lemma, since $u, v_1,\dots, v_m \in \Z^n \setminus \{ 0 \}$ and being non-positive is a closed condition.
	
	Thus the hypothesis of the lemma implies $0 \in \conv (-u, v_1,\dots, v_m)$, namely
	\begin{equation*}
		-au + \sum_{i=1}^m b_iv_i =0, \text{ for some } a, b_1,\dots, b_m\geq 0 \text{ with } a + \sum_{i=1}^m b_i=1.
	\end{equation*}
	We get the claimed result by showing $a \neq 0$. We first observe that $\cone (v_1,\dots, v_m)$ has an apex  at the origin, and hence $\conv (v_1,\dots, v_m)$ cannot contain $0$. Thus, again by the supporting hyperplane theorem, there exists an $\mathbb{R}$-linear map $g \co \R^n \to \R$ such that $g(v_1),\dots, g(v_m)$ are all strictly positive. We thus get, for the convex combination above, that
	\begin{equation*}
		-ag(u) + \sum_{i=1}^m b_i g(v_i) =0, \text{ where } a, b_1,\dots, b_m\geq 0 \text{ with } a + \sum_{i=1}^m b_i=1.
	\end{equation*}
	Since the above leads to an immediate contradiction when $a=0$, we get the claimed result.
\end{proof}

\begin{rema}
	We cannot drop the condition on the apex, as exhibited in the following example: $n=2$, $m=3$, and $u=(-1,0)$, $v_1 = (0,1)$, $v_2 = (0,-1)$, $v_3=(1,1)$. For this example (see Figure \ref{fig:conewoap}), $\cone (-u,v_1,v_2, v_3)  = \{ (x,y) \in \mathbb{R}^2 \mid x \ge 0 \}$ and satisfies the hypothesis of Lemma \ref{lemm:forchiss} but $u \not\in \cone (v_1,v_2, v_3)$. This cone does not have an apex at the origin, and the boundary of the cone agrees with the conical combination of $v_1$, $v_2$; the defining equation $f$ for the boundary satisfies $f(v_1) = f(v_2)=0$, which is non-positive regardless of the sign of $f(v_3)$ or $f(-u)$.
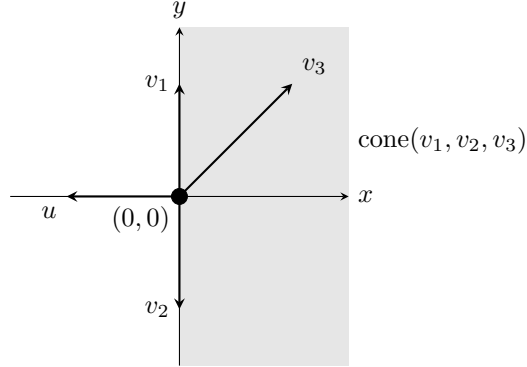
\begin{figure}
\centering
\begin{tikzpicture}[scale=1.5, >=stealth]

  \fill[lightgray!40] (0,-1.5) rectangle (1.5,1.5);

  \draw[->] (-1.5,0) -- (1.5,0) node[right] {$x$};
  \draw[->] (0,-1.5) -- (0,1.5) node[above] {$y$};

  \draw[->, thick] (0,0) -- (-1,0) node[below left] {$u$};
  \draw[->, thick] (0,0) -- (0,1) node[left] {$v_1$};
  \draw[->, thick] (0,0) -- (0,-1) node[left] {$v_2$};
  \draw[->, thick] (0,0) -- (1,1) node[above right] {$v_3$};
  \node[right] at (1.5,0.5) {$\cone (v_1,v_2, v_3)$};
  
  \filldraw[black] (0,0) circle (2pt);
  \node[below left] at (0,0) {$(0,0)$};

\end{tikzpicture}
\caption{Lemma \ref{lemm:forchiss} fails for a cone without an apex.}
\label{fig:conewoap}
\end{figure}
\end{rema}

We are ready to show Lemma \ref{lemm:chiss}.
\begin{proof}[Proof of Lemma \ref{lemm:chiss}]
Let $(z,w) \in \overline{M}$. Since $\Sigma = \cone (A_1,A_2,A_3,B_1,B_2,B_3)$ has the apex at the origin, so does its subcone $\sigma_{z,w}$. We now apply Lemma \ref{lemm:forchiss} to the generators of $\sigma_{z,w}$, with $n=2$. 
	
	Suppose first $(z,w) \in \overline{M}^{\chiss}$. Then the weight (\ref{dfeqhmwgt}) is nonnegative for any $\alpha =(\alpha_1,\alpha_2) \in \Z^2$, by the Hilbert--Mumford criterion. Noting that the definition (\ref{dfeqhmwgt}) naturally extends to any $\alpha =(\alpha_1,\alpha_2) \in \mathbb{Q}^2$ by clearing up the denominators, we find that for any generators $v_1 , \dots , v_m$ of the cone $\sigma_{z,w}$ and for any $\mathbb{R}$-linear map $f \co \mathbb{R}^2 \to \mathbb{R}$ with coefficients in $\mathbb{Q}$, at least one of $f(-C), f(v_1),\dots, f(v_m)$ is non-positive, since any nontrivial $\mathbb{Q}$-linear map $f \co \mathbb{Q}^n \to \mathbb{Q}$ can be written as $f (v) = \langle \alpha , v \rangle$ for some $\alpha \in \mathbb{Q}^2 \setminus \{ 0 \}$. Thus Lemma \ref{lemm:forchiss} implies $C \in \sigma_{z,w}$.
	
	Suppose conversely $C \in \sigma_{z,w}$. Again by Lemma \ref{lemm:forchiss} and by arguing as above, we find that the weight (\ref{dfeqhmwgt}) is nonnegative for any $\alpha \in \mathbb{Z}^2 \setminus \{ 0 \}$. The Hilbert--Mumford criterion concludes $(z,w) \in \overline{M}^{\chiss}$.
\end{proof}

For $\chi$-stability, we shall see isotropy subgroups at $(z,w) \in \overline{M}$. 
	Let $(z,w) \in \overline{M}$. Then the isotropy subgroup of $(\C^*)^2$ at $(z,w)$ coincides with 
	\begin{equation*}
		\{g \in (\C^*)^2 \mid \text{$g^{A_i}= 1$ for all $i$ such that $z_i \neq 0$ and $g^{B_j}=1$ for all $j$ such that $w_j \neq 0$}\}. 
	\end{equation*}
	Thus we have the following:
	\begin{lemm}\label{lemm:chis}
		Let $(z,w) \in \overline{M}$. Then the isotropy subgroup of $(\C^*)^2$ at $(z,w)$ is of dimension $0$ if and only if the $\mathbb{R}$-linear hull of $\sigma_{z,w}$ agrees with $\mathbb{R}^2$.
	\end{lemm}
	\begin{lemm}\label{lemm:chis2}
		$M \subset \overline{M}^{\chis}$ if and only if $C \in \Int \cone (A_i,B_j)$ for all  $i,j$ with $i \neq j$. 
	\end{lemm}

	\begin{proof}
		For the ``only if'' part, we assume that $M \subset \overline{M}^{\chis}$. Let $(i,j) \in \{1,2,3\}^2$ and assume that $i\neq j$. Then $(e_i,e_j) \in \C^3 \times \C^3$ sits in $M$, where $e_i$ and $e_j$ denote $i$-th and $j$-th standard basis vectors of $\C^3$, respectively. Since $(e_i,e_j) \in M \subset \overline{M}^{\chis} \subset \overline{M}^{\chiss}$, it follows from Lemma \ref{lemm:chiss} that $C \in \cone (A_i, B_j)$. It remains to prove $C \in \Int \cone (A_i,B_j)$. Suppose for contradiction that $C = aA_i$ for some non-negative real number $a$. Choosing an element $\alpha \in \mathbb{Q}^2$ orthogonal to $A_i$ in $\mathbb{Q}^2$, with an appropriate sign, we can find $\alpha \in \Z^2$ so that $\alpha \neq 0$, $\langle A_i,\alpha \rangle =0$ and $\langle B_j, \alpha \rangle \geq 0$, by clearing up the denominators. Let $\lambda$ be the $1$-PS given by $\lambda (h) = (h^{\alpha_1},h^{\alpha_2})$ for $h \in \C^*$. Then $\mu_\chi(\lambda, (e_i,e_j)) =-\min \{\langle A_i,\alpha\rangle, \langle B_j,\alpha\rangle, -\langle aA_i, \alpha\rangle\}= 0$, but $\lambda (\C^*)$ is nontrivial. This contradicts $M \subset \overline{M}^{\chis}$, as required, hence $C \notin \cone (A_i)$. By the same argument we also have $C \notin \cone (B_j)$. Therefore $C \in \Int \cone (A_i,B_j)$ by Lemma \ref{lemm:Intcone}, proving the ``only if'' part. 
		
		For the ``if'' part, we assume that $C \in \Int \cone (A_i,B_j)$ for all  $i,j$ with $i \neq j$. Let $(z,w) \in M$. Then there exists a pair of indices $(i',j')$ with $i'\neq j'$, such that $z_{i'} \neq 0$ and $w_{j'} \neq 0$. Indeed, if such a pair does not exist, we get
		\begin{equation} \label{eqziwjpf}
			z_i=0 \text{ or } w_j=0 \text{ for any } i,j=1,2,3 \text{ with } i \neq j.
		\end{equation}
		On the other hand, the assumption $z \neq 0$ and $w \neq 0$ implies that there exists $k,l \in \{ 1,2,3 \}$ such that $z_k \neq 0$ and $w_l \neq 0$. The condition (\ref{eqziwjpf}) implies $k=l$, and we also have $\sum_{i \neq k} z_i w_i= 0$ by applying (\ref{eqziwjpf}) to a pair of distinct indices in $\{ 1,2,3 \} \setminus \{ k \}$. We thus get $\sum_{i =1}^3 z_iw_i \neq 0$, which contradicts $(z,w) \in M$.
		
		With such $(i',j')$ given, Lemma \ref{lemm:chiss} and 
		\begin{equation*}
			C \in \Int \cone (A_{i'},B_{j'}) \subset \sigma_{z,w}
		\end{equation*}
		implies $(z,w) \in \overline{M}^{\chiss}$. By the Hilbert-Mumford criterion (see Section \ref{sec:HMC}),  $\mu_\chi(\lambda, (z,w)) \geq 0$ for any $1$-PS $\lambda$. It remains to show that the equality holds if and only if $\lambda(\C^*)$ is trivial. 
		Let $\alpha = (\alpha_1,\alpha_2) \in \Z^2$ and $\lambda$ the $1$-PS given by $\lambda(h) = (h^{\alpha_1},h^{\alpha_2})$ for $h \in \C^*$. If $\lambda (\C^*)$ is trivial, then $\alpha = 0$ and hence $\mu_\chi(\lambda, (z,w)) = 0$. Assume that $\lambda(\C^*)$ is nontrivial. Then $\alpha \neq 0$.  
		Since $C \in \Int \cone (A_{i'},B_{j'})$, we have that $A_{i'}$, $B_{j'}$ and $C$ are pairwise linearly independent. We take positive real numbers $a,b$ so that $aA_{i'} +bB_{j'} -C=0$. By taking inner product with $\alpha$, we have that among $\langle A_{i'}, \alpha\rangle$, $\langle B_{j'}, \alpha\rangle$, $\langle -C, \alpha\rangle$, at least one is negative and at least one is positive. Thus $\mu_\chi (\lambda, (z,w)) >0$. Namely, $\mu_\chi(\lambda, (z,w)) = 0$ if and only if $\lambda(\C^*)$ is trivial, as required.
	\end{proof}
	\begin{lemm}\label{lemm:stability}
		Assume that $M \subset \overline{M}^{\chis}$. Then $M = \overline{M}^{\chis} = \overline{M}^{\chiss}$.
	\end{lemm}
	\begin{proof}
		The inclusions $M \subset \overline{M}^{\chis} \subset \overline{M}^{\chiss}$ are obvious. Let $(z,w) \in \overline{M}^{\chiss}$. We show that $(z,w) \in M$, that is, $z\neq 0$ and $w \neq 0$. Assume that $w=0$. Then, it follows from Lemma \ref{lemm:chiss} that $C \in \cone \{A_i \mid z_i \neq 0\}$.  By Careth\'eodory's theorem, there exists a pair $(i_1,i_2) \in \{1,2,3\}^2$ with $i_1 \neq i_2$  such that $C \in \cone (A_{i_1}, A_{i_2})$. Let $j \in \{1,2,3\}$ be the index which is not $i_1$ and $i_2$. By the assumption and Lemma \ref{lemm:chis2}, we have $C \in \Int \cone (A_{i_1}, B_j)$ and $C \in \Int \cone (A_{i_2}, B_j)$. In particular, $C \notin \cone (A_{i_1}) \cup \cone (A_{i_2})$. By Lemma \ref{lemm:Intcone}, we have $C \in \Int \cone (A_{i_1}, A_{i_2})$. In particular, $A_{i_1}$ and $A_{i_2}$ form a basis of $\R^2$. We take positive real numbers $a_1,b_1,a_2, b_2$ so that $a_{1}A_{i_1} + b_1B_j = C$ and $a_{2}A_{i_2} + b_{2}B_j = C$. By eliminating $B_j$ from these equalities, we have $a_1b_2A_{i_1}-a_2b_1A_{i_2} = (b_2-b_1)C$. The coefficient $-a_2b_1$ is negative. This contradicts $C \in \Int \cone (A_{i_1}, A_{i_2})$ because $A_{i_1}$ and $A_{i_2}$ form a basis of $\R^2$. Thus $w \neq 0$. The same argument works for $z \neq 0$. Therefore $(z,w) \in M$, as required. 
	\end{proof}
	
\section{The main result and its application to biquotients of $SU(3)$}

The results we have established so far immediately yield the following.

\begin{theo}\label{theo:maintheo}
Let $A_1,A_2,A_3,B_1,B_2,B_3, C \in \Z^2$ be  such that $A_1 +B_1 = A_2+B_2 = A_3+B_3$.
	Consider  the action of $(\C^*)^2$ on $M$ whose weight is given by $A_1, A_2, A_3, B_1, B_2, B_3\in \Z^2$, and  a nontrivial linear character $\chi \co (\C^*)^2 \to \C^*$ whose weight is given by $\chi (g) = g^C$ for $C \in \Z^2$. Then, $\chi$ satisfies $M = \overline{M}^{\chis}$ if and only if it satisfies the  ``Japanese fan'' condition 
	\begin{itemize}
		\item[($\star$)] $\Sigma = \cone (A_1,A_2,A_3,B_1,B_2,B_3)$ has an apex at the origin and $C \in \Int \cone (A_i,B_j)$ for all  $i,j$ with $i \neq j$.  
	\end{itemize}
	Moreover, when the condition ($\star$) is satisfied, the quotient topological space $M / (\C^*)^2$ is a complex analytic space isomorphic  to the GIT quotient $\overline{M} \git_\chi (\C^*)^2$. In particular, $M / (\C^*)^2$ is compact. 
\end{theo}

\begin{proof}
	The first claim follows  directly from Lemmas \ref{lemm:chis2} and \ref{lemm:stability}.
	
	Proposition \ref{prop:dgzrc} implies that the scheme $\overline{M} \git_\chi (\C^*)^2 = \Proj (R^\chi)$ is projective over $R_0^\chi = \C$, which in particular implies that $\overline{M} \git_\chi (\C^*)^2$ is compact in the analytic topology. 
	
	Since $M = \overline{M}^{\chis}$, the action of $(\C^*)^2$ on $M$ is locally free. 
	As similar to  \cite[Section 4]{HMP1998}, Palais's characterization of the properness \cite{Palais61} and Luna's  slice theorem \cite{Luna} (see also \cite{Drezet2004}) imply that the action of $(\C^*)^2$ on $M$ is proper.
	 Therefore the quotient space $M/(\C^*)^2$ has a structure of a complex orbifold and hence an analytic space (see the appendix below). It follows from the universal  property and \cite[Proposition 5.5]{Simpson1994}\footnote{\cite[Proposition 5.5]{Simpson1994} is stated for semisimple groups, but the same argument works for reductive groups.} that $M/(\C^*)^2$ is isomorphic to $\overline{M} \git_\chi (\C^*)^2$ as analytic spaces. 
\end{proof}

Theorem \ref{theo:maintheo} has an  application to complex structures on $SU(3)$ and its biquotients. Let $T$ be the maximal compact torus $\{g = \diag (g_1,g_2,g_3) \mid g_1,g_2,g_3 \in S^1, g_1g_2g_3=1\}$ in $SU(3)$. Let $\rho_L, \rho_R \co (S^1)^2 \to T$ be smooth homomorphisms given by 
	\begin{equation*}
		\begin{split}
			\rho_L(t) &= \diag (t^{w_1^L}, t^{w_2^L},t^{w_3^L}) \\
			\rho_R(t) &= \diag (t^{w_1^R}, t^{w_2^R},t^{w_3^R})
		\end{split}
	\end{equation*}
	for $t \in (S^1)^2$, where $w_j^L, w_j^R \in \Z^2$. These homomorphisms give an action of $(S^1)^2$ on $SU(3)$ by $(t, g) \mapsto \rho_L(t)g\rho_R(t)^{-1}$ for $t \in (S^1)^2$ and $g \in SU(3)$. The quotient space $SU(3)/(\rho_L,\rho_R)((S^1)^2)$ of $SU(3)$ by this $(S^1)^2$-action is called a biquotient. 
	Put 
	\begin{equation*}
		A_j := w_j^L-w_1^R, \quad B_j := -w_j^L + w_3^R, \quad C := -w_1^R + w_3^R.
	\end{equation*}
	By \cite[Theorem 1.1]{IK2025}, if these elements satisfy the following condition 
	\begin{itemize}
		\item[($\star'$)] $C \notin \cone (A_i,A_j) \cup \cone (B_i,B_j)$ for all $i,j \in \{1,2,3\}$ and $C \in \cone (A_i,B_j)$ for all $i,j \in \{1,2,3\}$, 
	\end{itemize}
	then there exist a $T \times T$-invariant complex structure on $SU(3) $ and a  $T \times T/(\rho_L,\rho_R)((S^1)^2)$-invariant K\"ahler orbifold structure on the quotient $SU(3)/(\rho_L,\rho_R)((S^1)^2)$. 
	
	We briefly explain the complex structures on $SU(3)$ and $SU(3)/(\rho_L,\rho_R)((S^1)^2)$. The function $\Phi \co M \to \R^2$ defined as 
	\begin{equation*}
		\Phi (z,w) := \sum_{j=1}^3(A_j |z_j|^2+B_j |w_j|^2)
	\end{equation*}
	is a moment map for the $(S^1)^2$-action on $M$ given by 
	\begin{equation}\label{eq:cpttorusactonM}
		g\cdot (z,w):= (g^{A_1}z_1, g^{A_2}z_2, g^{A_3}z_3, g^{B_1}w_1, g^{B_2}w_2, g^{B_3}w_3)
	\end{equation} 
	for $g \in (S^1)^2$ and $(z,w) \in M$. Let $f_1, f_2 \co M \to \R$ be the first and second entry of $\Phi \co M \to \R^2$, respectively. Let $X$ and $Y$ be the Hamiltonian vector fields for $f_1$ and $f_2$, respectively. Let $J$ be the complex structure on $M$. Then we have commuting vector fields 
	\begin{equation*}
		Z := X-JY, \quad W := JX+Y
	\end{equation*}
	on $M$. These vector fields $Z$ and $W$ give a holomorphic action of $\C$ on $M$. More precisely, if $A_i = (a_{i1}, a_{i2})$ and $B_i = (b_{i1}, b_{i2})$, then the action of $\C$ on $M$ is given by 
	\begin{equation}\label{eq:CactiononM}
		u\cdot (z,w) = (e^{(a_{11}+\sqrt{-1}a_{12})u}, \dots, e^{(b_{31}+\sqrt{-1}b_{32})u})
	\end{equation}
	for $u \in \C$ and $(z,w) \in M$. One can see that the preimage $\Phi^{-1}(C) \subset M$ is equivariantly diffeomorphic to $SU(3)$ and each orbit of the $\C$-action on $M$ is transverse to the manifold $\Phi^{-1}(C)$. Using the holomorphic foliation obtained by this $\C$-action, we equip a complex structure on $\Phi^{-1}(C)$. The preimage $\Phi^{-1}(C)$ is invariant under the action of $(S^1)^2$ and orbits form a holomorphic foliation on $\Phi^{-1}(C)$. Thus quotient space $\Phi^{-1}(C)/(S^1)^2$ is a complex orbifold. The complex structures on $SU(3)$ and the biquotient $SU(3)/(\rho_L,\rho_R)((S^1)^2)$ are induced by the equivariant diffeomorphism between $SU(3)$ and $\Phi^{-1}(C)$. 
	
We remark that the condition ($\star'$) above is equivalent to the condition ($\star$) in Theorem \ref{theo:maintheo}. 
\begin{lemm} \label{lemm:cfptsr2}
	Let $A_1,A_2,A_3,B_1,B_2,B_3,C \in \R^2$. Then the following are equivalent: 
	\begin{enumerate}
		\item[($\star$)] $\Sigma = \cone (A_1,A_2,A_3,B_1,B_2,B_3)$ has an apex  at the origin and $C \in \Int \cone (A_i,B_j)$ for all  $i,j$ with $i \neq j$. 
		\item[($\star'$)] $C \notin \cone (A_i,A_j) \cup \cone (B_i,B_j)$ for all $i,j \in \{1,2,3\}$ and $C \in \cone (A_i,B_j)$ for all $i,j \in \{1,2,3\}$.
	\end{enumerate}
\end{lemm}

\begin{proof}
	The implication $(\star') \Rightarrow (\star)$ follows from \cite[Lemma 4.4]{IK2025} immediately. We show the opposite implication $(\star) \Rightarrow (\star')$. We first show that $C \notin \cone (A_i,A_j)$ for all $i,j \in \{1,2,3\}$. Since $C \in \Int \cone (A_i, B_j)$ for all $i,j$ with $i \neq j$, we have $C \notin \cone (A_i)$ for all $i \in \{1,2,3\}$. Take $i,j \in \{1,2,3\}$, with $i \neq j$, and suppose for contradiction that $C \in \cone (A_i,A_j)$. Then, $C \in \Int \cone (A_i,A_j)$ because $C \notin \cone (A_i) \cup \cone (A_j)$. By the same argument as the proof of Lemma \ref{lemm:stability}, we can see that $C \notin \Int \cone (A_i,A_j)$, yielding the desired contradiction. Therefore $C \notin \cone (A_i,A_j)$ for all $i,j$. By the same argument, we also have $C \notin \cone (B_i,B_j)$ for all $i,j$. 
	
	To prove the remaining claims, it suffices to show that $C \in \Int \cone (A_i,B_i)$ for all $i \in \{1,2,3\}$. Since $\Sigma$ has an apex at the origin, there exists a linear function $\alpha$ on $\R^2$ such that $\alpha(A_i)>0$, $\alpha(B_i)>0$ for all $i \in \{1,2,3\}$ and $\alpha (C)>0$. Since $C \notin \cone (A_1)$, $C$ and $A_1$ are linearly independent. Thus there exists a linear function $\beta$ on $\R^2$ such that $\beta (C)=0$ and $\beta (A_1)>0$. Then $\alpha$ and $\beta$ form a basis of the dual space of $\R^2$. Since $C \in \Int \cone (A_1,B_2)$ and $C \in \Int \cone (A_1,B_3)$, $\beta (C)=0$ and $\beta (A_1)>0$ imply that $\beta(B_2)<0$ and $\beta(B_3)<0$. By the same argument, we obtain $\beta (A_2), \beta(A_3)>0$ and $\beta (B_1)<0$. 
	In particular, $\alpha(A_i)\beta(B_i)-\alpha(B_i)\beta(A_i) <0$. By direct computation, $C$ is expressed as 
	\begin{equation*}
		C = \frac{\beta(B_i)}{\alpha(A_i)\beta(B_i)-\alpha(B_i)\beta(A_i)}A_i + \frac{-\beta(A_i)}{\alpha(A_i)\beta(B_i)-\alpha(B_i)\beta(A_i)}B_i.
	\end{equation*}
	Since the coefficients are all strictly positive, $C \in \Int \cone (A_i,B_i) \subset \cone (A_i,B_i)$. This completes the proof of Lemma.
\end{proof}

\begin{coro}\label{coro}
	Assume that $A_j, B_j$ ($j=1,2,3$) and $C$ satisfy the condition $(\star)$. Then, the following hold: 
	\begin{enumerate} 
		\item $SU(3)$ equipped with the above complex structure is biholomorphic to the quotient of $M$ by a free action of $\C$. 
		\item $SU(3)/(\rho_L,\rho_R)((S^1)^2)$ is isomorphic to $\overline{M} \git_\chi (\C^*)^2$ as analytic spaces. In particular, $SU(3)/(\rho_L,\rho_R)((S^1)^2)$ has a structure of a projective variety. 
	\end{enumerate}
\end{coro}
\begin{proof}
It follows from the definitions \eqref{eq:algtorusactiononM} and \eqref{eq:CactiononM} and of the actions on $M$ that the action of $\C$ on $M$  is nothing but the action of $(\C^*)^2$ on restricted to the subgroup $\C \to (\C^*)^2$ given by $u \mapsto (e^u, e^{\sqrt{-1}u})$. 
This together with the properness of the action of $(\C^*)^2$ on $M$ (see \cite[Section 4]{HMP1998}) yields that the action of $\C$ on $M$ is proper. In particular, the action of $\C$ on $M$ is free. Since $(\C^*)^2$ is an internal direct product of $(S^1)^2$ and $\C \to (\C^*)^2$, we have $M /(\C^*)^2 = (M/\C)/(S^1)^2$. This together with the compactness of $\overline{M}\git_{\chi} (\C^*)^2$ yields that $M/\C$ is compact. It implies that the inclusion $\Phi^{-1}(C) \hookrightarrow M$ induces an equivariant diffeomorphism $\Phi^{-1}(C) \to M/\C$.  Thus, $\Phi^{-1}(C)$ equipped with the complex structure is biholomorphic to $M/\C$, proving (1). 

 We have the following diagram: 
\begin{equation*}
	\begin{tikzcd}
		M \ar[r] & M/\C \ar[r]  & M/(\C^*)^2\\
		 & \Phi^{-1}(C) \ar[u] \ar[r] & \Phi^{-1}(C)/(S^1)^2 \ar[u] 
	\end{tikzcd}
\end{equation*}
The horizontal arrows are holomorphic quotient maps. Since the left vertical arrow is an equivariant biholomorphic map, the right vertical arrow is an isomorphism as complex orbifolds. 
By Theorem \ref{theo:maintheo}, $M/(\C^*)^2$ is isomorphic to $\overline{M}\git_\chi (\C^*)^2$ as analytic spaces. 
Therefore the biquotient $SU(3)/(\rho_L, \rho_R)((S^1)^2)$ has a structure of projective variety,  proving (2). 
\end{proof}

\section*{Appendix: The universal property of quotients by proper and locally free actions}
	Let $X$ be a complex manifold of dimension $m$ equipped with a holomorphic action of a complex Lie group $G$ of dimension $k$. 
	We assume that the action of $G$ on $X$ is proper and locally free. 
	The purpose of this appendix is to explain that the quotient $X/G$ carries natural structures of a complex orbifold and a complex analytic space, and that it satisfies the expected universal property in the category of complex analytic spaces. 
	More precisely, if $Y$ is an analytic space and $F \co X \to Y$ is a $G$-invariant morphism, then there exists a unique morphism $\widetilde{F} \co X/G \to Y$ such that $F = \widetilde{F} \circ \pi$, where $\pi \co X \to X/G$ is the natural projection. 
	
	We conclude by constructing a holomorphic orbifold atlas for $X/G$. 
	It follows from \cite[Theorem 1']{Nirenberg58} and the locally freeness of the action of $G$ on $X$ that there exist open neighborhoods $U^{(1)}$ at $x_0 \in X$, $V^{(1)}$ at $0 \in \C^{m-k}$, $W^{(1)}$ at $0 \in \C^k$ and a holomorphic chart $\phi^{(1)} = (\phi_{1}^{(1)},\phi_{2}^{(1)}) \co U^{(1)} \to V^{(1)}\times W^{(1)}$ of $X$ centered at $x_0$ such that 
	\begin{itemize}
		\item[(A1)] For each $v \in V^{(1)}$, the set $\{ y \in U^{(1)} \mid \phi_{1}^{(1)}(y) = v\}$ is a path-connected component of $U^{(1)} \cap G\cdot x$ for some $x \in U^{(1)}$ and conversely.  
	\end{itemize}
	Since the action of $G$ on $X$ is proper, the map $G \to X$, $g \mapsto g\cdot x_0$ descends to the closed embedding $G/G_{x_0} \to X$. 
	By taking $V^{(1)}$ sufficiently small, we may assume that $U^{(1)} \cap G\cdot x_0 = \{y \in U^{(1)} \mid \phi^{(1)}(y) = 0\}$.  
	By taking an inner product on $T_{x_0}X$ invariant under the action of $G_{x_0}$, we have a decomposition $T_{x_0}X = (T_{x_0}(G\cdot x_0))^\perp \oplus T_{x_0}(G\cdot x_0)$ as $G_{x_0}$-representations. 
	The differential 
	\begin{equation*}
		(d\phi^{(1)})_{x_0} \co T_{x_0}M \to T_0\C^m \cong \C^m = \C^{m-k} \times \C^k,
	\end{equation*}
	is an isomorphism. It follows from (A1) that the image of $T_{x_0}(G\cdot x_0)$ by $(d\phi^{(1)})_{x_0}$ coincides with $\{0\} \times \C^k$. 
	Choose a linear map $L \co \C^{m-k} \to \C^{k}$ so that the linear transformation $\psi \co \C^{m-k} \times \C^{m-k} \to \C^{m-k} \times \C^k$ given by $\psi(v,w) = (v, w + L(v))$ satisfies that the image of $(T_{x_0}(G\cdot x_0))^\perp$ by $\psi \circ (d\phi^{(1)})_0$ coincides with $\C^{m-k} \times \{0\}$. Define $\phi^{(2)} := (\phi_1^{(2)},\phi_2^{(2)}) := \psi \circ \phi^{(1)}$. Then there exist an open neighborhood $U^{(2)} \subset U^{(1)}$ at $x_0$, convex open neighborhoods $V^{(2)}$ at $0 \in \C^{m-k}$, $W^{(2)}$ at $0 \in \C^k$ such that the restriction $\phi^{(2)} \co U^{(2)} \to V^{(2)} \times W^{(2)}$ satisfies 
	\begin{itemize}
		\item[(A2)] For each $v \in V^{(2)}$, the set $\{ y \in U^{(2)} \mid \phi_{1}^{(2)}(y) = v\}$ is a path-connected component of $U^{(2)} \cap G\cdot x$ for some $x \in U^{(2)}$ and vice versa.
		\item[(B2)] The differential $(d\phi^{(2)})_{x_0} \co T_{x_0}X \to T_0\C^m \cong \C^m$ fits the decompositions  $T_{x_0}X = (T_{x_0}(G\cdot x_0))^\perp \oplus T_{x_0}(G\cdot x_0)$ and $\C^m = \C^{m-k} \times \C^k$. 
	\end{itemize}
	
	Put $U^{(3)} := U^{(2)}$. We define a biholomorphic map 
	\begin{equation*} 
		\begin{split}
		\phi^{(3)} &:= (\phi_1^{(3)},\phi_2^{(3)})\\
		&:= (d\phi^{(2)})_{x_0}^{-1} \circ \phi^{(2)}
		 \co U^{(2)} \to (d\phi^{(2)})_{x_0}^{-1} (V^{(2)} \times W^{(2)}) \\
		 & \subset T_{x_0}X = (T_{x_0}(G\cdot x_0))^\perp \oplus T_{x_0}(G\cdot x_0). 
		\end{split}
	\end{equation*}
	The holomorphic chart $\phi^{(3)}$ satisfies the following: 
	\begin{itemize}
		\item[(A3)] For each $v \in (T_{x_0}(G\cdot x_0))^\perp $, the set $\{ y \in U^{(3)} \mid \phi_{1}^{(3)}(y) = v\}$ is a path-connected component of $U^{(3)} \cap G\cdot x$ for some $x \in U^{(3)}$ and vice versa, unless empty.
		\item[(B3)] The differential $(d\phi^{(3)})_{x_0}$ of $\phi^{(3)}$ at $x_0$ is the identity map on $T_{x_0}X$. 
	\end{itemize}
	In fact, (A3) follows from (A2) and (B2). 
	Since the action of $G$ on $X$ is locally free, the isotropy subgroup $G_{x_0}$ is finite. Set $U^{(4)} = \bigcap _{g \in G_{x_0}} gU^{(3)}$. Clearly, $U^{(4)}$ is an invariant open neighborhood at $x_0$. Let $\phi^{(4)} \co U^{(4)} \to T_{x_0}X = (T_{x_0}(G\cdot x_0))^\perp \oplus T_{x_0}(G\cdot x_0)$ be the map defined by 
	\begin{equation*}
		\phi^{(4)} := (\phi_1^{(4)},\phi_2^{(4)}) := \frac{1}{|G_{x_0}|}\sum_{g \in G_{x_0}} g\circ \phi^{(3)} \circ g^{-1}. 
	\end{equation*}
	The map $\phi^{(4)}$ is holomorphic and $G_{x_0}$-equivariant. It satisfies $\phi^{(4)}(x_0) = 0$. Moreover (B3) yields the following:
	\begin{itemize}
		\item[(B4)] The differential $(d\phi^{(4)})_{x_0}$ of $\phi^{(4)}$ at $x_0$ is the identity map on $T_{x_0}X$. 
	\end{itemize}
	
	\begin{cla}
		Let $x_1, x_2 \in U^{(4)}$. If $x_2 \in U^{(4)}$ sits in the path-connected component of $U^{(4)} \cap G \cdot x_{1}$, then $\phi_1^{(4)}(x_1) = \phi_1^{(4)}(x_2)$. 
	\end{cla}
	\begin{proof}
		Let $i =1,2$. By definition of $\phi^{(3)}$, we have 
		\begin{equation*}
			\begin{split}
				\sum g\circ \phi^{(3)}\circ g^{-1}(x_i) &= \sum g\cdot \phi^{(3)}(g^{-1}\cdot x_i)\\
				&= \sum g\circ (\phi_1^{(3)} (g^{-1}\cdot x_i), \phi_2^{(3)} (g^{-1}\cdot x_i)) \\
				&= (\sum g\cdot \phi_1^{(3)}(g^{-1}\cdot x_i), \sum g\cdot \phi_2^{(3)}(g^{-1}\cdot x_i)). 
			\end{split}
		\end{equation*}
		Thus it suffices to show that $\phi_1^{(3)}(g^{-1}\cdot x_1) =  \phi_1^{(3)}(g^{-1}\cdot x_2)$ for all $g \in G_{x_0}$. Let $g \in G_{x_0}$. 
		By the assumption, there exists a path $x_t$, $t \in [1,2]$ such that $x_t \in U^{(4)} \cap G \cdot x$. Since $g^{-1}U^{(4)} \subset U^{(3)}$, the path $g^{-1}\cdot x_t$ belongs to $U^{(3)} \cap G \cdot (g^{-1}\cdot x_1)$. This together with (A3) yields that $\phi_1^{(3)}(g^{-1}\cdot x_1) =  \phi_1^{(3)}(g^{-1}\cdot x_2)$. This proves the claim.
	\end{proof}
	It follows from $(B4)$ and the implicit function theorem that there exists an open neighborhood $U^{(5)} \subset U^{(4)}$ at $x_0$ such that the map $\phi^{(4)}|_{U^{(5)}} \co U^{(5)} \to T_{x_0}X$ is a biholomorphic map onto its image. Set $U^{(6)} := \bigcap_{g \in G_{x_0}} U^{(5)}$. Then, the map $\phi^{(6)} := \phi^{(4)}|_{U^{(6)}} \co U^{(6)} \to T_{x_0}X$ is a $G_{x_0}$-equivariant biholomorphic map onto its image. Let $V^{(7)} \subset T_{x_0}X = (T_{x_0}(G\cdot x_0))^\perp$ be a connected $G_{x_0}$-invariant open neighborhood at $0$ and $W^{(7)} \subset  T_{x_0}(G\cdot x_0)$ be a connected $G_{x_0}$-invariant open neighborhood at $0$ such that $V^{(7)} \times W^{(7)} \subset \phi^{(6)}(U^{(6)})$. Finally, put $U^{(7)} := (\phi^{(6)})^{-1}(V^{(7)} \times W^{(7)})$ and $\phi^{(7)} := \phi^{(6)}|_{U^{(7)}} \co U^{(7)} \to V^{(7)} \times W^{(7)} \subset T_{x_0}X = (T_{x_0}(G\cdot x_0))^\perp \oplus T_{x_0}(G\cdot x_0)$. Then, $\phi^{(7)} \co U^{(7)} \to V^{(7)} \times W^{(7)} $ is a $G_{x_0}$-equivariant holomorphic chart such that
	\begin{itemize}
		\item[(A7)] For each $v \in V^{(7)}$, the set $\{ y \in U^{(7)} \mid \phi_{1}^{(7)}(y) = v\}$ is a path-connected component of $U^{(7)} \cap G\cdot x$ for some $x \in U^{(7)}$ and vice versa.
	\end{itemize}
	In fact, (A7) follows from the claim above, the connectedness of $W^{(7)}$ and the locally freeness of the action of $G$ on $X$. 

	We are now ready to construct a holomorphic slice $S$ through $x_0$. 
	For short, we denote by $\phi =(\phi_1,\phi_2) \co U \to V \times W$ instead of $\phi^{(7)} \co U^{(7)} \to V^{(7)} \times W^{(7)}$.  Define $S := \{ y \in U \mid \phi_2(y) = 0\} = \phi^{-1}(V \times \{0\})$ and define a map $F^S \co G \times S \to X$ by $F^S(g,s) = g\cdot s$ for $g \in G$ and $s \in S$. 
	Then $F^S$ is a holomorphic map and descends to the map $\widetilde{F^S} \co G \times_{G_{x_0}} S \to X$. By construction, $\widetilde{F^S}$ is $G$-equivariant and there exists a neighborhood $U_0 \subset G$  of the identity such that $\widetilde{F^S}|_{U_0 \times S}$ is a diffeomorphism onto its image.  Hence $\widetilde{F^S}$ is a covering map onto its image $GS = GU$. However the preimage of $x_0$ by $\widetilde{F}$ is the singleton $\{ [g, x_0] \mid g \in G_{x_0}\}$
	because $G\cdot x_0$ intersects with $S$ exactly one point $x_0$. Therefore $\widetilde{F^S}$ is a $G$-equivariant biholomorphic map onto its image $GU$.  Therefore, $S$ is a holomorphic slice through $x_0$.

	Consider all $x_0 \in X$ and all $G_{x_0}$-equivariant charts $\phi \co U \to V \times W$ as above. We obtain a holomorphic atlas $\{(U_\alpha, \phi_\alpha \co U_\alpha \to V_\alpha \times W_\alpha)\}_\alpha$ consisting of such charts. 
	Let $\pi \co X \to X/G$ be the natural projection. Since the action of $G$ on $X$ is proper, the quotient $X/G$ is Hausdorff. Since $\pi$ is open, we have that $X/G$ is second-countable and the collection $\{\pi(U_\alpha)\}_\alpha$ is an open base of $X/G$. Let $S_\alpha = \phi_\alpha^{-1} (V_\alpha \times \{0\})$ be the slice through the center $x_\alpha = \phi_{\alpha}^{-1}(0,0)$ of $\phi_\alpha \co U_\alpha \to V_\alpha \times W_\alpha$. Let $G_\alpha$ denote the isotropy subgroup at $x_\alpha$. Then, $\pi (U_\alpha)$ is homeomorphic to the quotient $S_\alpha/G_\alpha$ for each $\alpha$. We remark that $S_\alpha/G_\alpha$ is an analytic set because $G_\alpha$ is a finite group acting on $V_\alpha$ linearly. Let $\widetilde{\varphi}_\alpha \co V_\alpha \to \pi (U_\alpha)$ be the map given by $\widetilde{\varphi}_\alpha (v_\alpha) = \pi (\phi_\alpha^{-1}(v_\alpha,0))$ for $v_\alpha \in V_\alpha$. 
	Then $(\pi(U_\alpha), G_\alpha, \widetilde{\varphi}_\alpha \co V_\alpha \to \pi (U_\alpha))$ is an orbifold chart about $\pi (x_\alpha)$. The collection $\{(\pi (U_\alpha), G_\alpha, \widetilde{\varphi}_\alpha \co V_\alpha \to \pi (U_\alpha))\}_\alpha$ of all such orbifold charts is a holomorphic orbifold atlas on $X/G$, and $\{\pi (U_\alpha)\}_\alpha$ is an open base of $X/G$. Therefore the structure sheaf $\mathcal{O}_{X/G}$ of $X/G$ is determined by 
	\begin{equation*}
		\mathcal{O}_{X/G} (\pi (U_\alpha)) = \{ f \co \pi (U_\alpha) \to \C \mid \text{$\widetilde{\varphi}_\alpha^*f \co V_\alpha \to \C$ is holomorphic}\}. 
	\end{equation*}
	
	Let $\mathcal{O}_X$ be the structure sheaf of $X$. Let $Y$ be an analytic space and $\mathcal{O}_Y$ its structure sheaf. Let $F \co X \to Y$ be a $G$-invariant morphism as analytic spaces. Let $\widetilde{F} \co X/G \to Y$ be the map induced by $F$. Since $F$ is continuous, so is $\widetilde{F}$. We claim that $\widetilde{F}$ is a morphism of complex analytic spaces. Let $U_Y$ be an open subset of $Y$ and $g \in \mathcal{O}_Y(U_Y)$. For any chart $(U_\alpha, \widetilde{\phi}_\alpha = (\widetilde{\phi}_{\alpha1}, \widetilde{\phi}_{\alpha2}) \co U_\alpha \to V_\alpha \times W_\alpha)$ such that $F^{-1} (U_Y) \supset U_\alpha$, we have $(F^*g)|_{U_\alpha} \in \mathcal{O}_X (U_\alpha)$.  Since $F^*g$ is invariant under the action of $G$, there exists $h \in \mathcal{O}(V_\alpha)^{G_\alpha}$ such that $\widetilde{\phi}_{\alpha1}^* h = (F^*g)|_{U_\alpha}$, where $\mathcal{O}(V_\alpha)$ is the ring of holomorphic functions on $V_\alpha$. By pushing forward $h$ by $\widetilde{\varphi}_\alpha \co V_\alpha \to \pi(U_\alpha)$, we find $h' \in \mathcal{O}_{X/G}(\pi(U_\alpha))$ such that $\widetilde{\varphi}_\alpha^* h' = h$. Since $\widetilde{\varphi}_\alpha \circ \widetilde{\phi}_{\alpha_1} = \pi|_{U_\alpha}$, we have 
	\begin{equation*}
		\begin{split}
			(\pi|_{U_\alpha})^* h' &= \widetilde{\phi}_{\alpha1}^* \widetilde{\varphi}_\alpha^* h' \\
			&= \widetilde{\phi}_{\alpha1}^*h \\
			&= (F^*g)|_{U_\alpha}\\
			&= (\pi|_{U_\alpha})^*\widetilde{F}^*g. 
		\end{split}
	\end{equation*}
	Since $\pi|_{U_\alpha}$ is surjective, it follows that $(\pi|_{U_\alpha})^* h' = (\pi|_{U_\alpha})^*\widetilde{F}^*g$ implies $h' = (\widetilde{F}^*g)|_{\pi(U_\alpha)}$. In particular, $(\widetilde{F}^*g)|_{\pi(U_\alpha)} \in \mathcal{O}_{X/G}(\pi(U_\alpha))$. Therefore $\widetilde{F}^*g \in \mathcal{O}_{X/G}(\widetilde{F}^{-1}(U_Y))$. 
	Therefore, $\widetilde{F}$ is a morphism of complex analytic spaces, establishing  the desired universal property.

\section*{Declarations}
\subsection*{Funding}
The first author is supported by JSPS KAKENHI Grant Number JP23K03120 and JP24K00524. The second author is supported by JSPS KAKENHI Grant Number JP24K06742. The third author is supported by JSPS KAKENHI Grant Number JP24K00524.
\subsection*{Competing interests}
The authors have no competing interests to declare that are relevant to the content of this article.
\subsection*{Data Availability statement}
Data sharing not applicable to this article as no datasets were generated or analysed during the current study.


\begin{bibdiv}
\begin{biblist}

\bib{Anisimov2012-1}{article}{
   author={Anisimov, Artem},
   title={Spherical subgroups and double coset varieties},
   journal={J. Lie Theory},
   volume={22},
   date={2012},
   number={2},
   pages={505--522},
   issn={0949-5932},
   review={\MR{2976931}},
}

\bib{Anisimov2012-2}{article}{
   author={Anisimov, Artem},
   title={On existence of double coset varieties},
   journal={Colloq. Math.},
   volume={126},
   date={2012},
   number={2},
   pages={177--185},
   issn={0010-1354},
   review={\MR{2924248}},
   doi={10.4064/cm126-2-3},
}


\bib{AH2009}{article}{
   author={Arzhantsev, Ivan V.},
   author={Hausen, J\"urgen},
   title={Geometric invariant theory via Cox rings},
   journal={J. Pure Appl. Algebra},
   volume={213},
   date={2009},
   number={1},
   pages={154--172},
   issn={0022-4049},
   review={\MR{2462993}},
   doi={10.1016/j.jpaa.2008.06.005},
}
\bib{ADJ2013}{article}{
   author={Arzhantsev, Ivan V.},
   author={Celik, Devrim},
   author={Hausen, J\"urgen},
   title={Factorial algebraic group actions and categorical quotients},
   journal={J. Algebra},
   volume={387},
   date={2013},
   pages={87--98},
   issn={0021-8693},
   review={\MR{3056687}},
   doi={10.1016/j.jalgebra.2013.04.018},
}

\bib{ADHL}{book}{
    AUTHOR = {Arzhantsev, Ivan},
    AUTHOR = {Derenthal, Ulrich},
    AUTHOR = {Hausen, J\"{u}rgen},
    AUTHOR = {Laface, Antonio},
     TITLE = {Cox rings},
    SERIES = {Cambridge Studies in Advanced Mathematics},
    VOLUME = {144},
 PUBLISHER = {Cambridge University Press, Cambridge},
      YEAR = {2015},
     PAGES = {viii+530},
      ISBN = {978-1-107-02462-5},
   MRCLASS = {14Cxx (14Jxx 14Lxx)},
  review = {\MR{3307753}},
MRREVIEWER = {Alexandr V. Pukhlikov},
}



\bib{Drezet2004}{article}{
   author={Dr\'{e}zet, Jean-Marc},
   title={Luna's slice theorem and applications},
   conference={
      title={Algebraic group actions and quotients},
   },
   book={
      publisher={Hindawi Publ. Corp., Cairo},
   },
   isbn={977-5945-12-7},
   date={2004},
   pages={39--89},
   review={\MR{2210794}},
}

\bib{GKZ2020}{article}{
   author={Goertsches, Oliver},
   author={Konstantis, Panagiotis},
   author={Zoller, Leopold},
   title={Symplectic and K\"{a}hler structures on biquotients},
   journal={J. Symplectic Geom.},
   volume={18},
   date={2020},
   number={3},
   pages={791--813},
   issn={1527-5256},
   review={\MR{4142487}},
   doi={10.4310/JSG.2020.v18.n3.a6},
}


\bib{HMP1998}{article}{
   author={Heinzner, Peter},
   author={Migliorini, Luca},
   author={Polito, Marzia},
   title={Semistable quotients},
   journal={Ann. Scuola Norm. Sup. Pisa Cl. Sci. (4)},
   volume={26},
   date={1998},
   number={2},
   pages={233--248},
   issn={0391-173X},
   review={\MR{1631577}},
}


\bib{IK2020}{article}{
   author={Ishida, Hiroaki},
   author={Kasuya, Hisashi},
   title={Non-invariant deformations of left-invariant complex structures on compact Lie groups.},
   journal={ Forum Math. },
   volume={34},
   date={2022},
   number={4},
   pages={907--911},
   review={\MR{4445553 }},
   doi={10.1515/forum-2021-0133},
}

\bib{IK2025}{article}{
  author={Ishida, Hiroaki},
  author={Kasuya, Hisashi},
  title={{D}ouble sided torus actions and complex geometry on $SU (3)$},
  journal={J. Symplectic Geom.},
  volume={23}, 
  number={4}, 
  pages={751--777}, 
  year={2025}
}


\bib{Haefliger1985}{article}{
   author={Haefliger, A.},
   title={Deformations of transversely holomorphic flows on spheres and
   deformations of Hopf manifolds},
   journal={Compositio Math.},
   volume={55},
   date={1985},
   number={2},
   pages={241--251},
   issn={0010-437X},
   review={\MR{795716}},
}



\bib{King94}{article}{
    AUTHOR = {King, A. D.},
     TITLE = {Moduli of representations of finite-dimensional algebras},
   JOURNAL = {Quart. J. Math. Oxford Ser. (2)},
  FJOURNAL = {The Quarterly Journal of Mathematics. Oxford. Second Series},
    VOLUME = {45},
      YEAR = {1994},
    NUMBER = {180},
     PAGES = {515--530},
      ISSN = {0033-5606},
   MRCLASS = {16G10 (14D25)},
   review = {\MR{1315461}},
       DOI = {10.1093/qmath/45.4.515},
       URL = {https://doi.org/10.1093/qmath/45.4.515},
}

\bib{LMN2007}{article}{
   author={Loeb, Jean-Jacques},
   author={Manjar\'{\i}n, M\`onica},
   author={Nicolau, Marcel},
   title={Complex and CR-structures on compact Lie groups associated to
   abelian actions},
   journal={Ann. Global Anal. Geom.},
   volume={32},
   date={2007},
   number={4},
   pages={361--378},
   issn={0232-704X},
   review={\MR{2346223}},
   doi={10.1007/s10455-007-9067-7},
}

\bib{LV1997}{article}{
 author={L\'opez de Medrano, Santiago},
   author={Verjovsky, Alberto},
   title={A new family of complex, compact, non-symplectic manifolds},
   journal={Bol. Soc. Brasil. Mat. (N.S.)},
   volume={28},
   date={1997},
   pages={253--269},
   issn={},
   review={\MR{1479504}},
   doi={},
}

\bib{Luna}{article}{
   author={Luna, Domingo},
   title={Slices \'{e}tales},
   language={French},
   conference={
      title={Sur les groupes alg\'{e}briques},
   },
   book={
      series={Suppl\'{e}ment au Bull. Soc. Math. France, Tome 101},
      publisher={Soc. Math. France, Paris},
   },
   date={1973},
   pages={81--105},
   review={\MR{342523}},
   doi={10.24033/msmf.110},
}

\bib{Meersseman2000}{article}{
   author={Meersseman, Laurent}
   title={ A new geometric construction of compact complex manifolds in any dimension},
   journal={Math. Ann.},
   volume={317(1)},
   date={2000},
   pages={79--115},
   issn={},
   review={\MR{1760670}},
   doi={10.1007/s002080050360},
}

\bib{MV2004}{article}{
   author={Meersseman, Laurent},
   author={Verjovsky, Alberto},
   title={Holomorphic principal bundles over projective toric varieties},
   journal={J. Reine Angew. Math.},
   volume={572},
   date={2004},
   pages={57--96},
   issn={0075-4102},
   review={\MR{2076120}},
   doi={10.1515/crll.2004.054},
}


\bib{MFK}{book}{
    AUTHOR = {Mumford, D.},
    author = {Fogarty, J.},
    author = {Kirwan, F.},
     TITLE = {Geometric invariant theory},
    SERIES = {Ergebnisse der Mathematik und ihrer Grenzgebiete (2) [Results
              in Mathematics and Related Areas (2)]},
    VOLUME = {34},
   EDITION = {Third},
 PUBLISHER = {Springer-Verlag, Berlin},
      YEAR = {1994},
     PAGES = {xiv+292},
      ISBN = {3-540-56963-4},
   MRCLASS = {14D25 (58E05 58F05)},
  MRNUMBER = {\MR{1304906}},
MRREVIEWER = {Yi Hu},
}

\bib{Nakajima99}{book}{
    AUTHOR = {Nakajima, Hiraku},
     TITLE = {Lectures on {H}ilbert schemes of points on surfaces},
    SERIES = {University Lecture Series},
    VOLUME = {18},
 PUBLISHER = {American Mathematical Society, Providence, RI},
      YEAR = {1999},
     PAGES = {xii+132},
      ISBN = {0-8218-1956-9},
   MRCLASS = {14C05 (17B69)},
  review = {\MR{1711344}},
MRREVIEWER = {Mark Andrea A. de Cataldo},
       DOI = {10.1090/ulect/018},
       URL = {https://doi.org/10.1090/ulect/018},
}

\bib{Newstead78}{book}{
    AUTHOR = {Newstead, P. E.},
     TITLE = {Introduction to moduli problems and orbit spaces},
    SERIES = {Tata Institute of Fundamental Research Lectures on Mathematics
              and Physics},
    VOLUME = {51},
 PUBLISHER = {Tata Institute of Fundamental Research, Bombay; Narosa
              Publishing House, New Delhi},
      YEAR = {1978},
     PAGES = {vi+183},
      ISBN = {0-387-08851-2},
   MRCLASS = {14-02 (14D20)},
  review = {\MR{546290}},
MRREVIEWER = {G. Horrocks},
}

\bib{Nirenberg58}{article}{
	author={Nirenberg, Louis},
	review={Zbl 0099.37502},
	title={A complex Frobenius theorem},
	date={1958},
}

\bib{Palais61}{article}{
   author={Palais, Richard S.},
   title={On the existence of slices for actions of non-compact Lie groups},
   journal={Ann. of Math. (2)},
   volume={73},
   date={1961},
   pages={295--323},
   issn={0003-486X},
   review={\MR{126506}},
   doi={10.2307/1970335},
}







\bib{Simpson1994}{article}{
   author={Simpson, Carlos T.},
   title={Moduli of representations of the fundamental group of a smooth
   projective variety. I},
   journal={Inst. Hautes \'Etudes Sci. Publ. Math.},
   number={79},
   date={1994},
   pages={47--129},
   issn={0073-8301},
   review={\MR{1307297}},
}


\end{biblist}
\end{bibdiv}
\end{document}